\documentclass[11pt,a4paper]{amsart}
\usepackage[T1]{fontenc}
\usepackage[utf8]{inputenc}
\usepackage{lmodern,mathtools,amssymb}
\usepackage[margin=30mm]{geometry}
\usepackage{microtype}
\usepackage[hidelinks]{hyperref}

\newtheorem{theorem}{Theorem}[section]
\newtheorem{proposition}[theorem]{Proposition}
\newtheorem{lemma}[theorem]{Lemma}

\theoremstyle{remark}
\newtheorem{remark}[theorem]{Remark}
\numberwithin{equation}{section}
\newcommand{\N}{\mathbb N}
\newcommand{\F}{\mathbb F}
\newcommand{\one}{\mathbf 1}
\newcommand{\norm}[1]{\left\lVert#1\right\rVert}
\newcommand{\nt}[1]{\norm{#1}_{\theta}}

\DeclareMathOperator{\spn}{span}

\title[Semi-greedy Markushevich bases]{A semi-greedy Markushevich basis which is not quasi-greedy}
\author[P. M. Bern\'a]{Pablo M. Bern\'a}
\address{Universidad Tecnológica Atlántico Mediterráneo-UTAMED, Facultad de Empresa Digital, Tecnología y Derecho, Málaga, España}
\email{pablomanuel.berna@utamed.es}
\date{}
\thanks{The author was supported by Grant PID2022-142202NB-I00 (Agencia Estatal de Investigación, Spain).}

\begin{document}
\begin{abstract}
We construct a quasi-Banach space with a semi-greedy Markushevich basis which
is not quasi-greedy, and hence not almost greedy. This gives a negative answer
to the question of whether the equivalence between semi-greediness and almost
greediness for Markushevich bases in Banach spaces extends to the general
quasi-Banach setting.
\end{abstract}
\maketitle

\section{Introduction}

The thresholding greedy algorithm is based on a simple idea: given a vector,
we keep the coordinates with largest modulus and remove the others. This
procedure is easy to define and cheap to implement, but how well it works
depends strongly on the basis and on the geometry of the space. This leads to
several notions which measure different levels of performance of the greedy
algorithm, such as quasi-greedy, almost greedy and semi-greedy bases. We refer
to \cite{AABW} for the general theory of greedy approximation for
biorthogonal systems in quasi-Banach spaces.

There is a well-known connection between almost greediness and
semi-greediness. The first one was given in \cite{DKK} for Schauder bases in Banach spaces with finite cotype. For Schauder bases in general Banach spaces, the two notions are
equivalent; see \cite{Berna}. Later, Berasategui and Lassalle proved the same
equivalence for Markushevich bases in Banach spaces \cite{BL}. For Schauder
bases, the equivalence is also known in quasi-Banach spaces. Thus the natural
remaining question is whether the Schauder assumption can be removed in the
quasi-Banach setting. This question was explicitly raised in \cite{GarciaReview}.
It is less formal than it may first seem. Outside the locally convex setting,
some standard arguments involving partial sum operators, the dual space or
finite-dimensional minimization need extra care. Recent counterexamples also
show that greedy approximation may behave in a genuinely different way in
quasi-Banach spaces; see \cite{AABB}.

The aim of this paper is to answer this question in the negative. More
precisely, the Banach-space equivalence above does not survive for general
Markushevich bases in quasi-Banach spaces. We build a
quasi-Banach sequence space with a normalized, $M$-bounded Markushevich basis
which is semi-greedy but not quasi-greedy. Since almost greedy bases are
quasi-greedy, our basis is not almost greedy either. Thus, in this setting,
being able to obtain uniformly good Chebyshev approximations on greedy
supports does not imply that the thresholding greedy algorithm itself is
stable.

Let us explain where the space comes from. This is useful because the example
is not obtained by modifying a standard sequence space in an obvious way.
There are two properties that we want at the same time. First, for
semi-greediness we would like to move coefficients from a competing support
to a greedy support without paying too much in the quasi-norm. Second, to
break quasi-greediness we need some greedy projections to be much larger than
the vectors from which they come. So we look for a functional which is small,
or even zero, on balanced vectors of the form
\[
        \sum_{j=1}^r a_j(e_{n_j}-e_{k_j}),
\]
but becomes large on positive constant-coefficient vectors.

A convenient way to obtain these two behaviours is to use the defect of
additivity of the logarithmic map. More precisely, for $t>0$ we consider
\[
        h_t(z)=z\log\max\{|z|,t\}
\]
and, for $x\in\ell_1$, we set
\[
 F_t(x)=\sum_n h_t(x_n)-h_t\Big(\sum_nx_n\Big),
 \qquad H(x)=\sup_{t>0}|F_t(x)|.
\]
The truncation avoids the singular behaviour of the logarithm at zero, while
the correction involving the total sum makes the defect behave well under
scaling. The important point for us is that $H$ vanishes on disjoint pairs of
opposite coordinates, whereas it grows logarithmically on positive blocks.
This suggests defining
\[
        \|x\|_\theta=\|x\|_1+\theta H(x).
\]
The $\ell_1$ part keeps perturbations under control, and the logarithmic part
creates the gap needed to destroy quasi-greediness.

The basic computation already shows why this choice works. If $A$ and $B$ are
disjoint and $|A|=|B|=m$, then
\[
 \|\sum_{n\in A}e_n\|_\theta=m(1+\theta\log m),
 \qquad
 \|\sum_{n\in A}e_n-\sum_{j\in B}e_j\|_\theta=2m.
\]
Thus a positive block contains an extra logarithmic term, while a balanced
block does not. The semi-greedy proof uses exactly the cheap balanced
transfers in the second formula. The first formula, on the other hand,
produces greedy projections whose norms are arbitrarily large compared with
the original vectors. In this sense, the same feature of the space is
responsible for both sides of the example.

Our main result is the following.

\begin{theorem}\label{thm:main}
There exists a nonlocally convex quasi-Banach sequence space
$X\subset\ell_1$ with a normalized, $M$-bounded Markushevich basis
$(e_n)_{n\in\N}$ having the following properties:
\begin{enumerate}
\item The basis is semi-greedy, but is neither quasi-greedy nor almost greedy.
\item No ordering of $(e_n)$ is a Schauder basis of $X$.
\item The quasi-norm is invariant under permutations of the coordinates.
\item For every $0<p<1$, the space $X$ admits an equivalent $p$-norm. With
respect to that $p$-norm, the same system is a seminormalized, $M$-bounded,
semi-greedy Markushevich basis which is not quasi-greedy.
\end{enumerate}
More precisely, $X$ carries equivalent quasi-norms
$(\|\cdot\|_\theta)_{\theta>0}$ for which
\[
 \|e_n\|_\theta=\|e_n^*\|=1,
 \qquad C_{\mathrm{sg}}\le3+\frac{18\theta}{\mathrm e},
\]
where $C_{\mathrm{sg}}$ is the semi-greedy constant.
\end{theorem}

The paper is organized as follows. In Section~\ref{sec:preliminaries} we
recall the basic terminology on Markushevich bases and greedy approximation,
with particular attention to quasi-greedy, almost greedy and semi-greedy
systems. Section~\ref{sec:logspace} contains the construction of the space.
There we introduce the logarithmic defect functional and prove the estimates
that explain the two different behaviours needed in the example. In
Section~\ref{sec:semigreedy} we use the cheap balanced transfers built into
the quasi-norm to construct good Chebyshev approximants on greedy supports
and prove semi-greediness. Section~\ref{sec:failure} deals with the opposite
side of the construction: suitable constant-coefficient vectors show that
the canonical basis is not quasi-greedy. We also prove there that no
reordering of the canonical system is a Schauder basis and explain how to
pass to an equivalent $p$-norm for every $0<p<1$. Finally,
Section~\ref{sec:further} contains two additional remarks on coordinate tails
and on the canonical embedding into the bidual, which help to illustrate why
the usual Banach-space picture breaks down in this example.

\section{Preliminaries and notation}\label{sec:preliminaries}

Throughout the paper, $\F$ denotes either $\mathbb R$ or $\mathbb C$. We recall the definitions that will be used below. 

\subsection{Quasi-Banach spaces and biorthogonal systems}

A \emph{quasi-norm} on a vector space $X$ is a positive definite, absolutely
homogeneous map $\|\cdot\|:X\to[0,\infty)$ for which there is $\kappa\ge1$
such that
\[
 \|x+y\|\le \kappa(\|x\|+\|y\|),\qquad x,y\in X.
\]
A complete quasi-normed space is called a \emph{quasi-Banach space}. If
$0<p\le1$ and
\[
        \|x+y\|^p\le \|x\|^p+\|y\|^p,
\]
we say that $\|\cdot\|$ is a $p$-norm and that $X$ is a $p$-Banach space when
it is complete. We shall use the Aoki--Rolewicz theorem in the standard form
that every quasi-Banach space admits an equivalent $p$-norm for some
$0<p\le1$; a quantitative version will be used later.

A sequence $(e_n,e_n^*)_{n\in\N}\subset X\times X^*$ is a
\emph{biorthogonal system} if
\[
             e_n^*(e_k)=\delta_{n,k}\qquad(n,k\in\N).
\]
It is \emph{fundamental} if $\spn\{e_n:n\in\N\}$ is dense in $X$, and
\emph{total} if the functionals $(e_n^*)$ separate the points of $X$. A
fundamental and total biorthogonal system is called a \emph{Markushevich
basis}, or an \emph{$M$-basis}. Notice that this notion does not require the
series $\sum_ne_n^*(x)e_n$ to converge to $x$ in the given order. This is the
main distinction from a Schauder basis and is relevant for the example below.

We say that $(e_n)$ is \emph{seminormalized} if
\[
 0<\inf_n\|e_n\|\le \sup_n\|e_n\|<\infty,
\]
and that the biorthogonal system is \emph{$M$-bounded} if
\[
             \sup_n \|e_n\|\,\|e_n^*\|<\infty.
\]
For a seminormalized $M$-bounded system the coordinate functionals are
uniformly bounded. In particular, if the system is fundamental, then
$e_n^*(x)\to0$ for every $x\in X$: one first approximates $x$ by a finite
linear combination of the $e_n$ and then uses the uniform bound on the
coordinate functionals.

For a finite set $A\subset\N$ we write
\[
 E_A=\spn\{e_n:n\in A\},\qquad
 P_A(x)=\sum_{n\in A}e_n^*(x)e_n.
\]
When $A=\{1,\ldots,m\}$, $P_A$ is the usual $m$th partial sum operator for
the given ordering. A Markushevich basis is a \emph{Schauder basis} in that
ordering if $P_{\{1,\ldots,m\}}x\to x$ for every $x\in X$.

\subsection{Greedy sets and the thresholding greedy algorithm}

Let $x\in X$ and $m\in\N$. A set $A\subset\N$ with $|A|=m$ is a
\emph{greedy set of order $m$ for $x$} if
\begin{equation}\label{eq:greedy-set}
 \min_{n\in A}|e_n^*(x)|\ge
 \sup_{n\notin A}|e_n^*(x)|.
\end{equation}

Given a greedy set $A$ of $x$, the vector
\[
             G_A(x)=P_A(x)
\]
is a \emph{thresholding greedy approximant} of order $m$. Equivalently, one
may fix a greedy ordering $\rho:\N\to\N$ for $x$, satisfying
$|e_{\rho(1)}^*(x)|\ge |e_{\rho(2)}^*(x)|\ge\cdots$, and write
$G_m(x)=\sum_{j=1}^m e_{\rho(j)}^*(x)e_{\rho(j)}$. 
The system is \emph{quasi-greedy} if there is $C_q\ge1$ such that
\begin{equation}\label{eq:qg-def}
        \|x-P_Ax\|\le C_q\|x\|
\end{equation}
for every $x\in X$ and every greedy set $A$ of $x$. By the quasi-triangle
inequality, this residual formulation is equivalent, up to the constants, to
uniform boundedness of the greedy approximants $P_A(x)$ by $\|x\|$. Thus
quasi-greediness expresses stability of the thresholding greedy algorithm.

\subsection{Best $m$-term approximation and almost greedy bases}

For $m\in\N$ and $x\in X$, define
\begin{align}
 \sigma_m(x)
 &=\inf_{\substack{B\subset\N\text{ finite}\\ |B|\le m}}
      \inf_{z\in E_B}\|x-z\|, \label{eq:sigma}\\
 \widetilde\sigma_m(x)
 &=\inf_{\substack{B\subset\N\text{ finite}\\ |B|\le m}}
      \|x-P_B(x)\|. \label{eq:sigmatilde}
\end{align}
The first quantity is the error of best $m$-term approximation with arbitrary
coefficients, whereas the second only allows coordinate projections of $x$.
Clearly $\sigma_m(x)\le\widetilde\sigma_m(x)\le\|x\|$.

A basis is \emph{almost greedy} if there exists $C_a\ge1$ such that, for
every $x$, every $m$, and every greedy set $A$ of $x$ with $|A|=m$,
\begin{equation}\label{eq:ag-def}
       \|x-P_A(x)\|\le C_a\widetilde\sigma_m(x).
\end{equation}
Since the empty set is an admissible competitor in
\eqref{eq:sigmatilde}, almost greediness implies quasi-greediness.

\subsection{Chebyshev greedy approximation and semi-greedy bases}

The thresholding greedy algorithm fixes both the support and the coefficients
of the approximant. The Chebyshev greedy algorithm keeps the greedy support
but then chooses the coefficients optimally on that support. If $A$ is a
finite subset of $\N$, set
\begin{equation}\label{eq:cheb-error}
 \mathcal E_A(x)=\inf_{z\in E_A}\|x-z\|
 =\inf_{(a_n)_{n\in A}\in\F^A}
       \left\|x-\sum_{n\in A}a_ne_n\right\|.
\end{equation}
Whenever the infimum is attained, any minimizer will be denoted by
$\mathcal C_A(x)$ and called a \emph{Chebyshev greedy approximant} on $A$.
Uniqueness is neither assumed nor needed.

We say that the system is \emph{semi-greedy} if there is
$C_{\mathrm{sg}}\ge1$ such that
\begin{equation}\label{eq:sg-definition}
       \mathcal E_A(x)\le C_{\mathrm{sg}}\sigma_m(x),
       \qquad |A|=m,
\end{equation}
for every $x\in X$ and every greedy set $A$ of $x$. $C_{\mathrm{sg}}$ is called the semi-greedy constant.

In a normed space the
finite-dimensional minimization involved in \eqref{eq:cheb-error} is
standard. For a general quasi-norm, continuity should not be inserted
silently into the argument. In our example the relevant quasi-norm is lower
semicontinuous with respect to the ambient $\ell_1$ norm, and this will be
enough to prove directly that the infimum is attained on every finite set.

The distinction between \eqref{eq:qg-def} and \eqref{eq:sg-definition} is
central here. Quasi-greediness asks the original coefficients selected by the
greedy algorithm to be stable. Semi-greediness permits those coefficients to
be changed after the support has been selected. Our construction is designed
so that this freedom is decisive: a competitor supported away from a greedy
set can be transported to that set by adding balanced pairs, which are cheap
for the quasi-norm, even though the original greedy projection itself may be
very large.

In what follows, a \emph{sequence space} means a vector subspace of
$\F^{\N}$ with continuous coordinate functionals. 

\section{The logarithmic quasi-Banach space}\label{sec:logspace}

We now implement the idea described in the introduction. The logarithm is
useful because its failure of additivity is scale invariant after the linear
part is removed. The truncation at level $t$ prevents singular behaviour near
zero and, more importantly, allows us to take a supremum over all scales.
The correction term involving the total sum $S(x)$ makes the resulting
defect homogeneous. Finally, oddness of the scalar map forces exact
cancellation on opposite pairs of coordinates. These are precisely the three
features needed later: homogeneity, an $\ell_1$ control of the additive
defect, and zero cost for balanced transfers between supports.

All logarithms are natural. For $t>0$ and $z\in\F$, set
\[
 h_t(z)=z\log\max\{|z|,t\},\qquad h_t(0)=0.
\]
For $x=(x_n)\in\ell_1$, define
\[
 S(x)=\sum_{n=1}^{\infty}x_n,\qquad
 F_t(x)=\sum_{n=1}^{\infty}h_t(x_n)-h_t(S(x)),\qquad
 H(x)=\sup_{t>0}|F_t(x)|.
\]
For each fixed $t>0$, the series defining $F_t(x)$ converges
absolutely, since $|x_n|\le\norm{x}_1$ and the factors
$\log\max\{|x_n|,t\}$ are bounded uniformly in $n$. The value
$H(x)=+\infty$ is allowed at this stage.

\begin{lemma}\label{lem:defect}
For all $u,v\in\F$ and $t>0$,
\begin{equation}\label{eq:scalar-defect}
 |h_t(u+v)-h_t(u)-h_t(v)|
       \le\frac3{\mathrm e}(|u|+|v|).
\end{equation}
Consequently, with $L=6/\mathrm e$, for all $x,y\in\ell_1$,
\begin{equation}\label{eq:defect}
 |F_t(x+y)-F_t(x)-F_t(y)|
       \le L(\norm{x}_1+\norm{y}_1).
\end{equation}
\end{lemma}

\begin{proof}
Put $R=|u|+|v|$. The case $R=0$ is immediate. If $t\ge R$, all
three values in \eqref{eq:scalar-defect} are given by multiplication
by $\log t$, and the defect is zero. Otherwise, let $s=t/R\in(0,1)$.
The identity
\[
 h_t(Rz)=R h_s(z)+Rz\log R
\]
reduces the estimate to the unit disk, since the terms involving
$\log R$ cancel. For $|z|\le1$ we have
\[
 |z\log\max\{|z|,s\}|\le\frac1{\mathrm e}.
\]
Indeed, this follows from $r|\log r|\le1/\mathrm e$ for $0<r\le1$
when $|z|\ge s$, and from
$|z||\log s|\le s|\log s|\le1/\mathrm e$ otherwise. Apply this
bound to $u/R$, $v/R$, and $(u+v)/R$ to obtain
\eqref{eq:scalar-defect}.

Apply \eqref{eq:scalar-defect} to each coordinate of $x,y$ and to
$S(x),S(y)$. Summation, together with $|S(x)|\le\norm{x}_1$, gives
\eqref{eq:defect}.
\end{proof}

Fix $\theta>0$. On all of $\ell_1$, initially allowing infinite values,
put
\begin{equation}\label{eq:space-norm}
 \nt{x}=\norm{x}_1+\theta H(x),\qquad
 Y=\{x\in\ell_1:H(x)<\infty\}.
\end{equation}
The symbol $\norm{\cdot}_\theta$ denotes the quasi-norm of our space.
It is used to measure both approximation errors and competitor errors;
$H$ is only the functional entering its definition.

\begin{proposition}\label{prop:quasinorm}
The set $Y$ is a vector space, and $\norm{\cdot}_\theta$ is a
quasi-norm on $Y$. For $x,y\in Y$,
\begin{equation}\label{eq:refined-triangle}
 \nt{x+y}\le\nt{x}+\nt{y}
                 +\theta L(\norm{x}_1+\norm{y}_1).
\end{equation}
In particular,
\begin{equation}\label{eq:quasi-triangle}
 \nt{x+y}\le\kappa_\theta(\nt{x}+\nt{y}),
       \qquad \kappa_\theta=1+\theta L.
\end{equation}
\end{proposition}

\begin{proof}
For $\lambda\ne0$,
\[
 h_t(\lambda z)=\lambda h_{t/|\lambda|}(z)
                         +\lambda z\log|\lambda|.
\]
The linear terms cancel when forming $F_t$, whence
\[
 F_t(\lambda x)=\lambda F_{t/|\lambda|}(x),\qquad
 H(\lambda x)=|\lambda|H(x).
\]
Thus $Y$ is stable under scalar multiplication and
$\nt{\lambda x}=|\lambda|\nt{x}$. The assertion for $\lambda=0$
is immediate. Taking the supremum in \eqref{eq:defect} gives
\[
 H(x+y)\le H(x)+H(y)+L(\norm{x}_1+\norm{y}_1).
\]
Hence $Y$ is stable under addition. Adding the triangle inequality
in $\ell_1$ proves \eqref{eq:refined-triangle}; the inequality
$\norm{x}_1\le\nt{x}$ then gives \eqref{eq:quasi-triangle}.
Finally, $\nt{x}\ge\norm{x}_1$ proves positive definiteness.
\end{proof}

\begin{lemma}\label{lem:lsc}
For each $t>0$, the map $F_t:\ell_1\to\F$ is continuous. The
extended-valued functions $H$ and $\norm{\cdot}_\theta$ are lower
semicontinuous for the $\ell_1$-norm. In particular, if
$u^{(j)}\to u$ in $\ell_1$, then
\begin{equation}\label{eq:lsc}
 \nt{u}\le\liminf_{j\to\infty}\nt{u^{(j)}}.
\end{equation}
\end{lemma}

\begin{proof}
Fix $t,R>0$. The scalar function $h_t$ is Lipschitz on the disk
$\{z:|z|\le R\}$ with constant
\[
 K_{t,R}=1+\max\{|\log t|,|\log\max\{R,t\}|\}.
\]
To verify this, take $|z|=r\ge s=|w|$ and use
\[
 h_t(z)-h_t(w)
  =(z-w)\log\max\{r,t\}
       +w\log\frac{\max\{r,t\}}{\max\{s,t\}}.
\]
The second term has modulus at most $r-s\le|z-w|$. In fact,
\[
 s\log\frac{\max\{r,t\}}{\max\{s,t\}}\le r-s:
\]
for $s\ge t$ this follows from $\log a\le a-1$; for $s<t<r$ it
follows from $s\log(r/t)\le s(r-t)/t\le r-t\le r-s$; and for
$r\le t$ the left side vanishes. This also covers $s=0$.

If $\norm{x}_1,\norm{y}_1\le R$, the coordinates of $x,y$ and the
scalars $S(x),S(y)$ belong to that disk. Therefore
\begin{align*}
 |F_t(x)-F_t(y)|
 &\le\sum_n|h_t(x_n)-h_t(y_n)|
                +|h_t(S(x))-h_t(S(y))|\\
 &\le2K_{t,R}\norm{x-y}_1.
\end{align*}
This proves continuity for fixed $t$; no uniform bound in $t$ is
needed. Since $H$ is a supremum of continuous functions $|F_t|$, it
is lower semicontinuous. The same holds for
$\norm{\cdot}_1+\theta H$, proving \eqref{eq:lsc}, including its
extended-valued formulation.
\end{proof}

\begin{proposition}\label{prop:complete}
The space $(Y,\norm{\cdot}_\theta)$ is complete.
\end{proposition}

\begin{proof}
Let $(x^{(j)})_{j\ge1}$ be a Cauchy sequence in
$(Y,\norm{\cdot}_\theta)$. Since
$\|u\|_1\le\|u\|_\theta$, the sequence is Cauchy in $\ell_1$ and hence
there is $x\in\ell_1$ such that $x^{(j)}\to x$ in $\ell_1$.

Fix $j$. Applying the lower semicontinuity from Lemma~\ref{lem:lsc} to
$x^{(j)}-x^{(k)}\to x^{(j)}-x$ in $\ell_1$ gives
\[
 \|x^{(j)}-x\|_\theta
 \le \liminf_{k\to\infty}\|x^{(j)}-x^{(k)}\|_\theta.
\]
In particular, the right-hand side is finite, so $x^{(j)}-x\in Y$ and,
since $x^{(j)}\in Y$, also $x\in Y$. Given $\varepsilon>0$, choose $j_0$
such that
$\|x^{(j)}-x^{(k)}\|_\theta<\varepsilon$ whenever $j,k\ge j_0$.
For $j\ge j_0$, the preceding inequality yields
\[
             \|x^{(j)}-x\|_\theta\le\varepsilon.
\]
Thus $x^{(j)}\to x$ in $\|\cdot\|_\theta$, and $Y$ is complete.
\end{proof}

Every finitely supported vector belongs to $Y$: its function $F_t$
is constant for sufficiently small $t$, vanishes for sufficiently
large $t$, and is continuous in between. We may therefore define
\begin{equation}\label{eq:X}
 X=\overline{c_{00}}^{\,\norm{\cdot}_\theta}\subset Y.
\end{equation}

\begin{proposition}\label{prop:basis}
The space $(X,\norm{\cdot}_\theta)$ is quasi-Banach. Its canonical
vectors, together with the coordinate functionals $e_n^*(x)=x_n$,
form a Markushevich basis satisfying
\[
 \nt{e_n}=\norm{e_n^*}=1\qquad(n\in\N).
\]
Moreover, coordinate permutations act isometrically on $X$ and, also, all $\theta>0$
give the same space $X$ and equivalent quasi-norms.
\end{proposition}

\begin{proof}
By Proposition~\ref{prop:complete}, $X$ is a closed subspace of a
complete quasi-normed space, hence is complete. Since $H(e_n)=0$,
we have $\nt{e_n}=1$. Also,
\[
 |e_n^*(x)|\le\norm{x}_1\le\nt{x}.
\]
This proves continuity and $\norm{e_n^*}\le1$; evaluation at $e_n$
gives equality. Biorthogonality is immediate, density of the linear
span is part of \eqref{eq:X}, and the coordinate functionals separate
points because $X$ is a subspace of $\ell_1$. For a nonempty finite
$A$, repeated application of \eqref{eq:defect} also gives
\[
 \norm{u}_1\le\nt{u}
 \le\bigl(1+\theta L(|A|-1)\bigr)\norm{u}_1
       \qquad(u\in E_A).
\]
Indeed, each one-coordinate vector has defect zero, and every partial
sum has $\ell_1$-norm at most $\norm{u}_1$. Thus the topology on
$E_A$ is the usual finite-dimensional topology.

Every coordinate permutation preserves $\norm{\cdot}_1$, $S$, and
each $F_t$, and maps $c_{00}$ onto itself. It therefore acts
isometrically on $X$. Finally, for $\theta,\eta>0$,
\begin{equation}\label{eq:equivalent}
 \min\{1,\theta/\eta\}\norm{x}_\eta
 \le\nt{x}\le\max\{1,\theta/\eta\}\norm{x}_\eta.
\end{equation}
Thus the closures in \eqref{eq:X} agree as sets and as topological
vector spaces.
\end{proof}

\section{Chebyshev approximation on greedy supports}\label{sec:semigreedy}

The vanishing of $H$ on disjoint pairs of opposite coordinates is the
key to semi-greediness.

\begin{lemma}\label{lem:pairs}
Let $i_1,j_1,\ldots,i_k,j_k$ be distinct indices and
$c_1,\ldots,c_k\in\F$. If
\[
 d=\sum_{r=1}^k c_r(e_{i_r}-e_{j_r}),
\]
then $F_t(d)=0$ for every $t>0$, and
$\nt{d}=\norm{d}_1=2\sum_{r=1}^k|c_r|$.
\end{lemma}

\begin{proof}
The function $h_t$ is odd. Thus the contributions of each pair cancel
in $\sum_nh_t(d_n)$, while $S(d)=0$.
\end{proof}

For clarity, write $\sigma_m^\theta$ and $\mathcal E_A^\theta$ for
the approximation errors measured in $\norm{\cdot}_\theta$.

\begin{lemma}\label{lem:minimum}
For every $x\in X$ and every finite $A\subset\N$, there exists
$\mathcal C_A^\theta(x)\in E_A$ such that
\[
 \mathcal E_A^\theta(x)
   =\min_{(a_n)_{n\in A}\in\F^A}
           \nt{x-\sum_{n\in A}a_ne_n}
   =\nt{x-\mathcal C_A^\theta(x)}.
\]
\end{lemma}

\begin{proof}
Let $(z_k)\subset E_A$ be a minimizing sequence. Its errors are
bounded, and
\[
 \norm{z_k}_1\le\norm{x}_1+\norm{x-z_k}_1
               \le\norm{x}_1+\nt{x-z_k}.
\]
Hence the coefficient vectors of $z_k$ are bounded in the
finite-dimensional space $\F^A$. After passing to a subsequence,
$z_k\to z$ in $\ell_1$ for some $z\in E_A$. Lemma~\ref{lem:lsc}
gives
\[
 \nt{x-z}\le\liminf_{k\to\infty}\nt{x-z_k}
                  =\mathcal E_A^\theta(x).
\]
The reverse inequality follows from the definition of the infimum.
\end{proof}

\begin{theorem}\label{thm:sg}
The canonical Markushevich basis of $(X,\norm{\cdot}_\theta)$ is
semi-greedy. For every $x\in X$, every $m\ge1$, and every greedy
set $A$ of $x$ with $|A|=m$,
\begin{equation}\label{eq:sg-bound}
 \min_{(a_n)_{n\in A}\in\F^A}
       \nt{x-\sum_{n\in A}a_ne_n}
 \le3(1+\theta L)\sigma_m^\theta(x),\qquad L=6/\mathrm e.
\end{equation}
\end{theorem}

\begin{proof}
Fix $x,m,A$ as in the statement, and take an arbitrary competitor
\[
 y=\sum_{j\in B}b_je_j,\qquad |B|\le m.
\]
Put $r=x-y$. Since $|B\setminus A|\le|A\setminus B|$, choose an
injection $\pi:B\setminus A\to A\setminus B$. Move the
coefficients of $y$ outside $A$ to their corresponding positions in
$A\setminus B$ by defining
\[
 z_y=\sum_{j\in A\cap B}b_je_j
          +\sum_{j\in B\setminus A}b_je_{\pi(j)}\in E_A.
\]
The difference
\[
 d=y-z_y=\sum_{j\in B\setminus A}b_j(e_j-e_{\pi(j)})
\]
has disjoint pairs of opposite coordinates, so
$H(d)=0$ and $\nt{d}=\norm{d}_1$ by Lemma~\ref{lem:pairs}.

For $j\in B\setminus A$, greediness and $\pi(j)\notin B$ give
\[
 |x_j|\le|x_{\pi(j)}|=|r_{\pi(j)}|.
\]
As $b_j=x_j-r_j$, it follows that
\begin{align*}
 \norm{d}_1
 &=2\sum_{j\in B\setminus A}|b_j|\\
 &\le2\sum_{j\in B\setminus A}(|x_j|+|r_j|)\\
 &\le2\sum_{j\in B\setminus A}(|r_{\pi(j)}|+|r_j|)
 \le2\norm{r}_1.
\end{align*}
The last inequality uses injectivity of $\pi$ and disjointness of
$B\setminus A$ and $\pi(B\setminus A)$. Since $x-z_y=r+d$,
the refined triangle inequality \eqref{eq:refined-triangle} yields
\begin{align*}
 \nt{x-z_y}
 &\le\nt{r}+\norm{d}_1
                  +\theta L(\norm{r}_1+\norm{d}_1)\\
 &\le\nt{r}+(2+3\theta L)\norm{r}_1\\
 &\le3(1+\theta L)\nt{x-y}.
\end{align*}
Consequently,
\[
 \mathcal E_A^\theta(x)
      \le3(1+\theta L)\nt{x-y}
\]
for every competitor $y$ with at most $m$ terms. The left side is
independent of $y$, so taking the infimum on the right gives
$\mathcal E_A^\theta(x)\le3(1+\theta L)\sigma_m^\theta(x)$.
Lemma~\ref{lem:minimum} identifies the left side with the minimum
in \eqref{eq:sg-bound}. The constant is independent of $x$, $m$,
and the choice of greedy set, including all ties.
\end{proof}

\section{Failure of quasi-greediness and Schauder orderings}\label{sec:failure}

Given a semi-normalized M-basis $(e_n)$ and a finite set $A$, we write $\one_A:=\sum_{n\in A}e_n$.
Two elementary computations exhibit the different behavior of
positive and balanced constant-coefficient vectors. If $|A|=m\ge1$,
then
\[
 F_t(\one_A)
     =m\log\max\{1,t\}-m\log\max\{m,t\}.
\]
Its maximal modulus is $m\log m$, and hence
\begin{equation}\label{eq:positive}
 \nt{\one_A}=m(1+\theta\log m).
\end{equation}
For disjoint $A,B$ of cardinality $m$, Lemma~\ref{lem:pairs} gives
\begin{equation}\label{eq:balanced}
 \nt{\one_A-\one_B}=2m.
\end{equation}

\begin{theorem}\label{thm:notqg}
The canonical basis of $X$ is not quasi-greedy and therefore is not
almost greedy. This failure is witnessed by vectors having unique
greedy sets at the relevant orders.
\end{theorem}

\begin{proof}
For disjoint $A_m,B_m$ with $|A_m|=|B_m|=m$, let
\[
 f_m=\left(1+\frac1m\right)\one_{A_m}-\one_{B_m}.
\]
The set $A_m$ is the unique greedy set of $f_m$ of cardinality $m$,
and $f_m-P_{A_m}f_m=-\one_{B_m}$. Moreover, $S(f_m)=1$, so
\[
 F_t(f_m)=(m+1)\log
          \frac{\max\{1+1/m,t\}}{\max\{1,t\}}.
\]
This is nonnegative and reaches its maximum for $0<t\le1$.
Therefore
\begin{align*}
 \nt{f_m}&=2m+1+\theta(m+1)\log(1+1/m),\\
 \nt{f_m-P_{A_m}(f_m)}&=m(1+\theta\log m).
\end{align*}
It follows that
\begin{equation}\label{eq:divergence}
 \frac{\nt{f_m-P_{A_m}(f_m)}}{\nt{f_m}}
 =\frac{m(1+\theta\log m)}
        {2m+1+\theta(m+1)\log(1+1/m)}
 \longrightarrow\infty.
\end{equation}
Thus there is no uniform quasi-greedy estimate. Since almost
greediness implies quasi-greediness, it also fails.
\end{proof}

\begin{proposition}\label{prop:structure}
The space $X$ is not locally convex, and no ordering of its canonical
vectors is a Schauder basis.
\end{proposition}

\begin{proof}
Suppose that $X$ is locally convex. There is a convex neighborhood
$U$ of zero contained in $\{x:\nt{x}\le1\}$. Since $U$ is a
neighborhood, there is $\delta>0$ such that $\delta e_n\in U$ for
all $n$. Convexity gives $(\delta/m)\one_A\in U$ whenever
$|A|=m$. By \eqref{eq:positive}, this implies
\[
 \delta(1+\theta\log m)=\nt{(\delta/m)\one_A}\le1
       \qquad(m\ge1),
\]
which is impossible.

Now fix any ordering $(e_{\pi(n)})_{n\in\N}$ and put
$A_m=\{\pi(1),\ldots,\pi(m)\}$ and
$B_m=\{\pi(m+1),\ldots,\pi(2m)\}$. Equations
\eqref{eq:positive} and \eqref{eq:balanced} imply
\[
 \norm{P_{A_m}}
 \ge\frac{\nt{P_{A_m}(\one_{A_m}-\one_{B_m})}}
          {\nt{\one_{A_m}-\one_{B_m}}}
 =\frac{1+\theta\log m}{2}.
\]
The partial sum projections are thus not uniformly bounded. To
justify the implication in the quasi-Banach setting, suppose that
$P_{A_m}x\to x$ for every $x\in X$. Then $\sup_m\nt{P_{A_m}x}$
is finite for every $x$, and the closed sets
\[
 D_k=\{x\in X:\nt{P_{A_m}x}\le k\text{ for every }m\},
       \qquad k\in\N,
\]
cover $X$. Closedness follows from Lemma~\ref{lem:lsc} and continuity
of each finite coordinate projection into $\ell_1$. A quasi-Banach
space is completely metrizable (by Aoki--Rolewicz) and hence is a
Baire space. Some $D_k$ therefore contains a neighborhood of a point
$x_0$. Choose $\delta>0$ such that $x_0+u\in D_k$ whenever
$\nt{u}<\delta$. The quasi-triangle inequality gives
\[
 \nt{P_{A_m}u}\le
 \kappa_\theta\bigl(\nt{P_{A_m}(x_0+u)}+\nt{P_{A_m}x_0}\bigr)
 \le2\kappa_\theta k.
\]
Homogeneity would imply $\sup_m\norm{P_{A_m}}\le
4\kappa_\theta k/\delta$, a contradiction. Thus the chosen ordering
is not Schauder, and the ordering was arbitrary.
\end{proof}

\begin{proposition}\label{prop:pnorm}
For every $0<p<1$, the space $X$ admits an equivalent $p$-norm.
Under this renorming the canonical system remains a seminormalized,
$M$-bounded, semi-greedy Markushevich basis which is not quasi-greedy.
\end{proposition}

\begin{proof}
Fix $0<p<1$ and choose
\[
 \theta_p=\frac1L\left(2^{1/p-1}-1\right)>0.
\]
By \eqref{eq:quasi-triangle}, the quasi-norm
$\norm{\cdot}_{\theta_p}$ has concavity constant at most
$2^{1/p-1}$. By the Aoki--Rolewicz theorem, $X$ admits a $p$-norm
$\norm{\cdot}_{(p)}$ equivalent to $\norm{\cdot}_{\theta_p}$. Thus there are
constants $a_p,b_p>0$ such that
\[
 a_p\norm{x}_{\theta_p}\le\norm{x}_{(p)}
       \le b_p\norm{x}_{\theta_p},\qquad x\in X.
\]
By \eqref{eq:equivalent}, this $p$-norm is equivalent to every
$\norm{\cdot}_\theta$, $\theta>0$. In particular, the canonical vectors remain
seminormalized and the coordinate functionals remain uniformly bounded.

For completeness, suppose
$a\nt{x}\le\norm{x}_{(p)}\le b\nt{x}$, with $a,b>0$.
The greedy sets depend only on the fixed coordinate functionals, so
they are unchanged. Taking infima over the same competitors gives
\[
 \mathcal E_A^{(p)}(x)
 \le b\mathcal E_A^\theta(x)
 \le bC_{\mathrm{sg}}\sigma_m^\theta(x)
 \le\frac{b}{a}C_{\mathrm{sg}}\sigma_m^{(p)}(x).
\]
The Chebyshev infimum is attained for the $p$-norm: a minimizing
sequence is bounded in $E_A$, has a convergent subsequence there,
and $p$-norms are continuous. Thus this estimate gives
semi-greediness in the minimum formulation as well. The lower bound
by $a$ and upper bound by $b$ preserve the divergence in
\eqref{eq:divergence}. Completeness, density, separation by the
coordinate functionals, seminormalization, and $M$-boundedness are
also preserved by equivalence of quasi-norms.
\end{proof}

\begin{proof}[Proof of Theorem~\ref{thm:main}]
Proposition~\ref{prop:basis} supplies the space and the normalized,
$M$-bounded Markushevich basis, together with permutation invariance.
Theorem~\ref{thm:sg} gives semi-greediness with the stated constant,
whereas Theorem~\ref{thm:notqg} excludes quasi-greediness and almost
greediness. Proposition~\ref{prop:structure} proves nonlocal
convexity and the absence of a Schauder ordering. Finally,
Proposition~\ref{prop:pnorm} establishes the assertion for every
$0<p<1$.
\end{proof}

\section{Further remarks}\label{sec:further}

The construction has two additional features which help to explain why the usual Banach-space intuition is misleading here. First, distances to coordinate tails do not recover the quasi-norm. Second, although the coordinate functionals separate points, the bidual norm only sees the underlying $\ell_1$ part.

\begin{remark}\label{rem:tails}
The example also shows a concrete obstruction to recovering the
quasi-norm from distances to tails. Let
\[
 T_N=\overline{\spn\{e_n:n>N\}}^{\,\norm{\cdot}_\theta},
 \qquad q_\theta(u)=\sup_N\inf_{v\in T_N}\nt{u-v}.
\]
For $|A|=m\ge1$ we have
\[
 m\le q_\theta(\one_A)\le2m.
\]
For the lower bound, choose $N\ge\max A$. Every element of $T_N$
vanishes on $A$, because the coordinate functionals are continuous,
and $\norm{\cdot}_\theta\ge\norm{\cdot}_1$. For the upper bound,
given $N$, choose a set $B$ of cardinality $m$, disjoint from $A$,
whose indices exceed $N$, and apply \eqref{eq:balanced} to the
competitor $\one_B\in T_N$. On the other hand,
$\nt{\one_A}=m(1+\theta\log m)$. Hence no constant $C$ can satisfy
$\nt{u}\le Cq_\theta(u)$ for all $u\in X$.
\end{remark}

\begin{remark}\label{rem:bidual}
The canonical map into the bidual also loses the logarithmic term.
More precisely, for each $\theta>0$, the dual of
$(X,\norm{\cdot}_\theta)$ identifies isometrically with $\ell_\infty$
under the pairing
\[
 f_b(x)=\sum_n b_nx_n\qquad(b\in\ell_\infty,\ x\in X).
\]
Indeed, $|f_b(x)|\le\norm{b}_\infty\norm{x}_1
\le\norm{b}_\infty\nt{x}$, while evaluation at $e_n$ shows
$\norm{f_b}=\norm{b}_\infty$. Conversely, if $f\in X^*$, then
$b_n=f(e_n)$ is bounded, and $f=f_b$ by continuity and density of
$c_{00}$ in $X$. Hence the canonical map $J:X\to X^{**}$ satisfies
\[
 \norm{Jx}_{X^{**}}
 =\sup_{\norm{b}_\infty\le1}\left|\sum_n b_nx_n\right|
 =\norm{x}_1.
\]
In particular, $J$ is injective but is not an isomorphic embedding:
for $|A|=m$, we have $\norm{J\one_A}_{X^{**}}=m$, whereas
$\nt{\one_A}=m(1+\theta\log m)$. Thus separation by continuous
linear functionals does not provide a norming dual for the
quasi-norm used in the semi-greedy estimate.
\end{remark}

\textbf{Acknowledgment}. The author would like to thank Miguel Berasategui for several helpful discussions during the preparation of this paper.

\end{document}